\documentclass[11pt]{article}
\usepackage{tocbibind} 
\usepackage{amsmath,amsthm,amssymb}
\usepackage{fancyhdr}
\usepackage[british]{babel}
\usepackage{geometry,mathtools}
\usepackage{enumitem}
\usepackage{algpseudocode}
\usepackage{dsfont}
\usepackage{centernot}
\usepackage{xstring}
\usepackage{colortbl}
\usepackage{thm-restate}

\usepackage[section]{algorithm}
\usepackage{graphicx}
\usepackage[font={footnotesize}]{caption}
\usepackage[usenames,dvipsnames,table]{xcolor}
\usepackage{pstricks,tikz}
\usetikzlibrary{patterns,calc,fadings,decorations.pathreplacing}
\usepackage[h]{esvect}
\usepackage[
		bookmarksopen=true,
		bookmarksopenlevel=1,
		colorlinks=true,
		linkcolor=Blue,
        linktoc=page,
		citecolor=Blue,
]{hyperref}
\usepackage{bbm}
\usepackage{xskak}

\usepackage{amsfonts}

\newcommand{\overbar}[1]{\mkern 1.5mu\overline{\mkern-1.5mu#1\mkern-1.5mu}\mkern 1.5mu}

\theoremstyle{plain}
\newtheorem{theorem}{Theorem}
\newtheorem{lemma}{Lemma}[section]
\newtheorem{claim}[lemma]{Claim}
\newtheorem{proposition}[lemma]{Proposition}

\newtheorem{corollary}[theorem]{Corollary}
\newtheorem{corollary*}[lemma]{Corollary}
\newtheorem{conjecture}[theorem]{Conjecture}

\newtheorem{definition}[lemma]{Definition}

\title{\vspace{-0.9cm} On Ramsey size-linear graphs and related questions}

\author{
Domagoj Brada\v{c}\thanks{Department of Mathematics, ETH, Z\"urich, Switzerland. Research supported in part by SNSF grant 200021\_196965. Email: \textbf{\{domagoj.bradac, lior.gishboliner, benjamin.sudakov\}@math.ethz.ch}.}
\and 
Lior Gishboliner\footnotemark[1]
\and Benny Sudakov\footnotemark[1]
}

\date{}

\begin{document}
\maketitle 


\begin{abstract}
    In this paper we prove several results on Ramsey numbers $R(H,F)$ for a fixed graph $H$ and a large graph $F$, in particular for $F = K_n$. These results extend earlier work of Erd\H{o}s, Faudree, Rousseau and Schelp and of Balister, Schelp and Simonovits on so-called Ramsey size-linear graphs. Among others, we show that if $H$ is a subdivision of $K_4$ with at least $6$ vertices, then $R(H,F) = O(v(F) + e(F))$ for every graph $F$. We also conjecture that if $H$ is a connected graph with $e(H) - v(H) \leq \binom{k+1}{2} - 2$, then $R(H,K_n) = O(n^k)$. The case $k=2$ was proved by Erd\H{o}s, Faudree, Rousseau and Schelp. We prove the case $k=3$. 
\end{abstract}

\section{Introduction}
For two graphs $H$ and $F$, the Ramsey number $R(H,F)$ is the smallest $N$ such that for every graph $G$ on $N$ vertices, either $G$ contains a copy of $H$ or its complement $\overbar{G}$ contains a copy of $F$.
One of the central problems in graph Ramsey theory is the estimation of Ramsey numbers of complete graphs $R(K_s,K_n)$ for $s$ fixed and large $n$. The classical Erd\H{o}s-Szekeres \cite{ES} theorem implies that $R(K_s,K_n) = O(n^{s-1})$, and this was improved to $R(K_s,K_n) = O(n^{s-1}/\log^{s-2} n)$ by a celebrated result of Ajtai, Koml\'{o}s and Szemer\'{e}di \cite{AKS}. As for lower bounds, Spencer \cite{Spencer} showed that $R(H,K_n) = \tilde{\Omega}(n^{m_2(H)})$
\footnote{As customary, for two functions $f,g$, we write $g(n) = \tilde\Omega(f(n))$ to mean that $g(n) \geq f(n)/\text{polylog}(n)$, $g(n) = \tilde{O}(f(n))$ to mean that $g(n) \leq f(n) \cdot \text{polylog}(n)$, and $g(n) = \tilde{\Theta}(f(n))$ to mean that $f(n)/\text{polylog}(n) \leq g(n) \leq f(n) \cdot \text{polylog}(n)$.} for every graph $H$, where $m_2(H)$ is the $2$-density\footnote{The $2$-density $m_2(H)$ is defined as the maximum of $\frac{e(H')-1}{v(H')-2}$ over all subgraphs $H'$ of $H$ with at least $3$ vertices.} of $H$. This in particular implies that $R(K_s,K_n) = \tilde{\Omega}(n^{(s+1)/2})$. Kim \cite{Kim} improved the implied logarithmic term in the case $s = 3$, obtaining the tight result $R(K_3,K_n) = \Theta(n^2/\log n)$. This was later generalized by Bohman and Keevash \cite{BK}, who improved the logarithmic term for every $s$. On the other hand, no improvement to the exponent of $n$ has been obtained for any $s \geq 4$. Very recently, Mubayi and Verstra\"ete \cite{MV} showed that the existence of optimally-dense pseudorandom $K_s$-free graphs would imply that $R(K_s,K_n) = \tilde{\Omega}(n^{s-1})$, matching the upper bound. This gives some evidence to the conjecture that $R(K_s,K_n) = \tilde{\Theta}(n^{s-1})$ for every $s$. 

A more general problem is to estimate $R(H,K_n)$ for an arbitrary graph $H$. It is well-known that $R(H,K_n) = O(n)$ if and only if $H$ is a forest. In fact, when $H$ is a tree, a classical result of Chv\'atal \cite{Chvatal} gives the exact value of $R(H,K_n)$. It is thus natural to ask which graphs satisfy $R(H,K_n) = O(n^k)$ for $k \geq 2$. Erd\H{o}s, Faudree, Rousseau and Schelp \cite{EFRS} were the first to study this problem, proving several results for the case $k=2$. They
proved that $R(H,K_n) = O(n^2)$ for every connected graph $H$ with $e(H) - v(H) \leq 1$. This result is tight, as $e(K_4) - v(K_4) = 2$ and $R(K_4,K_n) = \tilde{\Omega}(n^{5/2})$ (by the aforementioned result of Spencer \cite{Spencer}). We propose the following conjecture which generalizes the result of Erd\H{o}s, Faudree, Rousseau and Schelp.
\begin{conjecture}\label{conj:e-v}
Let $k \geq 1$. For every connected graph $H$ with $e(H) - v(H) \leq \binom{k+1}{2} - 2$, it holds that $R(H,K_n) = O(n^k)$. 
\end{conjecture}
Note that $e(K_{k+2}) - v(K_{k+2}) = \binom{k+1}{2} - 1$. Hence, if the aforementioned conjecture that $R(K_s,K_n) = \tilde{\Omega}(n^{s-1})$ is true, then the constant $\binom{k+1}{2} - 2$ in Conjecture \ref{conj:e-v} would be best possible. In this paper we prove the first open case of Conjecture \ref{conj:e-v}, namely the case $k = 3$. 
\begin{theorem}\label{thm:v-e, k=3}
    Let $H$ be a connected graph with $e(H) - v(H) \leq 4$. Then $R(H,K_n) = O(n^3)$.
\end{theorem}
In the proof of Theorem \ref{thm:v-e, k=3}, we make use of the following claim, which bounds the Ramsey number $R(H,K_n)$ in terms of the treewidth of $H$. This might be of independent interest. Recall that a graph is called a {\em $k$-tree} if it is $K_{k+1}$ or if it is obtained from a smaller $k$-tree by adding a new vertex and connecting it to $k$ vertices which form a clique. The treewidth of $H$ is the minimal $k$ for which $H$ is a subgraph of a $k$-tree.
\begin{restatable}{proposition}{restateproptw} \label{prop:tw}
For every fixed graph $H$, we have $R(H,K_n) = O(n^{\text{tw}(H)})$.
\end{restatable}
In addition to proving Conjecture \ref{conj:e-v} in the case $k = 3$, we show that the full conjecture holds if $K_n$ is replaced with $K_{n,n}$ (see Proposition \ref{prop:e-v, complete bipartite}). Along the way we also obtain bounds for $R(H,K_{n,n})$ for graphs $H$ with bounded maximum degree or bounded degeneracy; see Corollaries \ref{cor:max degree} and \ref{cor:degenerate}.

Note that if $R(H,K_n) = O(n^2)$ then this can be written as $R(H,K_n) = O(e(K_n))$. This motivated Erd\H{o}s, Faudree, Rousseau and Schelp \cite{EFRS} to define the so-called {\em Ramsey size-linear graphs}. A graph $H$ is called Ramsey size-linear if 
\begin{equation}\label{eq:Ramsey size-linear}
R(H,F) = O(e(F))    
\end{equation}
holds for every graph $F$ with no isolated vertices. This notion was introduced in \cite{EFRS}, where the authors established some basic results and raised several intriguing questions. In particular, Erd\H{o}s, Faudree, Rousseau and Schelp asked whether it is true that every graph $H$ with $m_2(H) \leq 2$ is Ramsey size-linear. This would imply that every $2$-degenerate graph is Ramsey size-linear. These questions seem to be still out of reach at the moment. Perhaps in light of this, Erd\H{o}s et al. also asked about specific graphs $H$. In particular, they asked whether $K_4^*$, the graph obtained from $K_4$ by subdividing one edge, is Ramsey size-linear. This question was later reiterated by Balister, Schelp and Simonovits \cite{BSS}. While we cannot supply an affirmative answer, we can show that \eqref{eq:Ramsey size-linear} at the very least holds for every {\bf bipartite} graph $F$.
\begin{theorem}\label{thm:vs bipartite}
For every bipartite graph $F$ with no isolated vertices, it holds that $R(K_4^*,F) = O(e(F))$. 
\end{theorem}
The above question of Erd\H{o}s et al.~for $K_4^*$ motivates the study of Ramsey numbers for subdivisions of $K_4$. Balister, Schelp and Simonovits \cite{BSS} showed (as part of a more general result) that the graph obtained from $K_4$ by subdividing an edge four times is Ramsey size-linear. Here we extend this further, showing that {\em every} subdivision of $K_4$ other than $K_4^*$ is Ramsey size linear.
\begin{theorem}\label{thm:K4 subdivisions}
Every subdivision of $K_4$ on at least $6$ vertices is Ramsey size-linear. 
\end{theorem}

\begin{figure}
    \centering
    \includegraphics{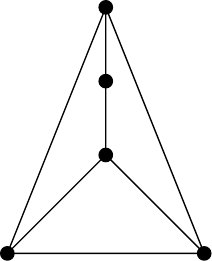}
    \caption{$K_4^*$}
    \label{fig:K_4^*}
\end{figure}

Theorem \ref{thm:K4 subdivisions} is used in the proof of Theorem \ref{thm:v-e, k=3}. Indeed, it is well-known that a graph $G$ has treewidth larger than $2$ if and only if $G$ contains a $K_4$-subdivision. Combining Theorem \ref{thm:K4 subdivisions} with Proposition \ref{prop:tw} and some additional arguments gives Theorem \ref{thm:v-e, k=3}.

The proofs of Theorems \ref{thm:vs bipartite} and \ref{thm:K4 subdivisions} heavily rely on the use of averaging arguments. Theorem~\ref{thm:K4 subdivisions} additionally uses dependent random choice (see e.g., \cite{FS} for a description of this method and a brief history) which can be viewed as a more sophisticated use of averaging and convexity arguments.

The rest of this short paper is organized as follows. Section \ref{sec:prelim} contains some lemmas used in the proofs of Theorems \ref{thm:vs bipartite} and \ref{thm:K4 subdivisions}. We then prove Theorem \ref{thm:vs bipartite} in Section \ref{sec:vs bipartite} and Theorem \ref{thm:K4 subdivisions} in Section \ref{sec:K4 subdivisions}. Section \ref{sec:e-v} contains all results related to Conjecture \ref{conj:e-v}. Finally, the last section includes some comments and related open questions. We use $\log n$ to denote the natural logarithm of $n$. We omit floor and ceiling signs whenever these are not crucial. We use $\delta(G), \Delta(G), d(G)$ to denote the minimum, maximum and average degree of $G$, respectively. We will frequently use the fact that a graph with $n$ vertices and average degree $d$ contains an independent set of size at least $n/(d+1)$. This is a well-known consequence of Tur\'an's theorem. 

\section{Preliminary Lemmas}\label{sec:prelim}


\begin{lemma}\label{lem:greedy empty bipartite graph}
Let $r > 0$. Consider a bipartite graph with sides $X,Y$ with $e(X,Y) \leq r|Y|$. Then there are $X' \subseteq X, Y'\subseteq Y$ such that $|X'| \geq \lfloor |X|/(r+1) \rfloor, |Y'| \geq \lfloor |Y|/(r+1) \rfloor$ and there are no edges between $X'$ and $Y'$. 
\end{lemma}
\begin{proof}
By averaging, there is $X'\subseteq X$ with $|X'| = \lfloor |X|/(r+1) \rfloor$ such that $e(X',Y) \leq e(X,Y)/(r+1) \leq r|Y|/(r+1)$. Hence, there are at least $\lfloor |Y|/(r+1) \rfloor$ vertices $y \in Y$ which have no edge to $X'$. Take $Y'$ to be the set of these vertices.
\end{proof}

In the proof of Theorem \ref{thm:vs bipartite}, it is convenient to assume that the host graph is (almost) regular. The following lemma allows us to assume that the maximum degree is larger than the average degree by no more than a logarithmic factor. 
\begin{lemma}\label{lem:regularization}
    Let $G$ be a graph on $N \geq 12$ vertices. Then there is an induced subgraph $G'$ of $G$ with average degree $d(G')$ such that $|V(G')| \geq N/6$ and $\Delta(G') \leq d(G') \cdot \log |V(G')|$.
\end{lemma}
\begin{proof}
We run the following process for $2N/3$ steps. If the current graph $G$ satisfies 
$\Delta(G) \leq d(G) \cdot \log |V(G)|$, then stop. Otherwise, take $v \in V(G)$ with 
$d_{G}(v) > d(G) \cdot \log |V(G)|$, and replace $G$ with $G - v$. 
Letting $d$ denote the old average degree and $d_{\text{new}}$ the new average degree, we have 
\begin{align*}
d_{\text{new}} &= \frac{d |V(G)| - 2d_{G}(v)}{|V(G)| - 1}
<
\frac{d|V(G)| - 2d\log |V(G)|}{|V(G)| - 1} = 
d \cdot \left( 1 - \frac{2\log |V(G)| - 1}{|V(G)| - 1} \right) \\&\leq 
d \cdot \left( 1 - \frac{2\log N - 1}{N - 1} \right),
\end{align*}
where the last inequality holds because the function $x \mapsto \frac{2\log x-1}{x-1}$ is decreasing for $x \geq 4$, say, and $|V(G)| \geq N/3 \geq 4$ (as we only run the process for $2N/3$ steps and in each step remove one vertex).
So we see that if the process did not stop, then the average degree of the final graph $G$ is at most 
$$(N-1) \cdot \left( 1 - \frac{2\log N - 1}{N - 1} \right)^{2N/3} \leq 
(N-1) \cdot \exp\left( {-\frac{2N/3 \cdot (2\log N - 1)}{N-1}} \right) \leq (N-1) \cdot e^{-\log N} < 1.$$
Therefore, the final graph $G$ contains an empty subgraph on at least $|V(G)|/2 = N/6$ vertices. This subgraph satisfies the assertion of the lemma. 
\end{proof}

The following lemma shows that if some number of highest-degree vertices of a graph $F$ has already been embedded into the complement $\overbar{G}$ of a graph $G$, and if $N = |V(G)|$ is large enough compared to the average degree of $G$ and the degrees of the vertices already used in the embedding, then one can complete the embedding of $F$ into $\overbar{G}$. The proof uses a basic greedy embedding argument. 

\begin{lemma}\label{lem:greedy embedding}
Let $F$ be a graph with $m$ edges. 
Let $0 \leq k \leq |V(F)|$, and let $A$ be the set of the $k$ highest-degree vertices in $F$. Let $G$ be a graph with $N$ vertices and average degree $d$. Let $\sigma' : A \rightarrow V(G)$ be an embedding of $F[A]$ into $\overbar{G}$. Suppose that 
\begin{equation}\label{eq:greedy embedding}
N \geq \frac{4m}{k+1} \cdot \max \left( \max_{v \in \sigma'(A)}{d_G(v)}, \; 2d \right) + 2|V(F)|.
\end{equation}
Then there is an embedding $\sigma$ of $F$ into $\overbar{G}$ which extends $\sigma'$ (i.e. $\sigma(x) = \sigma'(x)$ for every $x \in A$). 
\end{lemma}
\begin{proof}
Let $W$ be the set of $v \in V(G)$ with $d_G(v) \leq 2d$. Then $|W| \geq N/2$. We embed the vertices of $V(F) \setminus A$ one-by-one into $W$. Let $x \in V(F) \setminus A$. We want to choose $\sigma(x) \in W$ which is different from all previously embedded vertices, such that if $y \in V(F)$ has already been embedded and is adjacent in $F$ to $x$, then $\sigma(x)$ is not adjacent in $G$ to $\sigma(y)$. 
If $y \in A$, then the degree of $\sigma(y)$ in $G$ is of course not larger than $\max_{v \in \sigma'(A)}{d_G(v)}$. And if $y \notin A$, then $\sigma(y) \in W$ and hence its degree in $G$ is at most $2d$. 
The total number of vertices $y$ which we need to consider is $d_F(x) \leq 2m/(k+1)$, where the inequality holds by the choice of $A$, as $|A| = k$ and $e(F) = m$.  
So in total, the number of vertices which {\bf cannot} play the role of $\sigma(x)$ is at most 
$\frac{2m}{k+1} \cdot \max \left( \max_{v \in \sigma'(A)}{d_G(v)}, \; 2d \right) + |V(F)| - 1 < N/2 \leq |W|$. Hence, there is a suitable choice for $\sigma(x) \in W$. 
\end{proof}

\begin{corollary*}\label{cor:degeneracy embedding}
Let $F$ be a graph with $m$ edges, and let $G$ be a graph with $N$ vertices and average degree $d \leq \sqrt{\frac{N^2 - 2N \cdot |V(F)|}{48m}}$. Then $\overbar{G}$ contains a copy of $F$. 
\end{corollary*}
\begin{proof}
Let $W$ be the set of $v \in V(G)$ with $d_G(v) \leq 2d$. Then $\Delta(G[W]) \leq 2d$ and $|W| \geq N/2$. Hence, $G[W]$ contains an independent set $I$ of size $k \geq \frac{N/2}{2d + 1} \geq \frac{N}{6d}$. Let $A$ be the set of $\min\{k,|V(F)|\}$ highest-degree vertices of $F$. Mapping $F[A]$ arbitrarily into $I$ gives an embedding of $F[A]$ into $\overbar{G}$, since $I$ is independent in $G$. If $A = V(F)$ then we are done. Else, we apply Lemma \ref{lem:greedy embedding} to complete the embedding of $F$ into $\overbar{G}$. We only need to verify the condition \eqref{eq:greedy embedding}. Since $I \subseteq W$, all vertices in $I$ have degree at most $2d$ in $G$. So for \eqref{eq:greedy embedding} to hold, it suffices that 
$N \geq \frac{4m}{N/6d} \cdot 2d + 2|V(F)|$, which holds by the assumption of the lemma.
\end{proof}

\noindent
Next, we need a bipartite version of Lemma \ref{lem:greedy embedding}.

\begin{lemma}\label{lem:bipartite greedy embedding}
Let $F$ be a bipartite graph with sides $A,B$ and $m$ edges. 
Let $k,\ell \geq 0$, let $A'$ be the set of the $k$ highest-degree vertices in $A$, and let $B'$ be the set of the $\ell$ highest-degree vertices in $B$. 
Let $G$ be a graph with $N$ vertices and average degree $d$. Let $\sigma' : A'\cup B'\rightarrow V(G)$ be an embedding of $F[A' \cup B']$ into $\overbar{G}$. Suppose that
\begin{enumerate}
    \item $B' = B$ or $N \geq \frac{2m}{\ell+1} \cdot \max_{v \in \sigma'(A')}{d_G(v)} + 2|V(F)|$.
    \item $A' = A$ or $N \geq \frac{m}{k+1} \cdot \max \left( \max_{v \in \sigma'(B')}{d_G(v)}, \; 2d \right) + |V(F)|$. 
\end{enumerate}
Then there is an embedding $\sigma$ of $F$ into $\overbar{G}$ which extends $\sigma'$. 
\end{lemma}
\begin{proof}
Let $W$ be the set of $v \in V(G)$ with $d_G(v) \leq 2d$. Then $|W| \geq N/2$. We will embed the vertices of $V(F) \setminus (A'\cup B')$ one-by-one. We first embed the vertices of $B \setminus B'$ into $W$. Let $b \in B \setminus B'$. We want to choose $\sigma(b) \in W$ such that $\sigma(b)$ is not adjacent in $G$ to $\sigma'(a)$ for any $a \in A'$ with $(a,b) \in E(F)$, and such that $\sigma(b)$ is different from all previously embedded vertices. 
We have $d_F(b) \leq m/(\ell+1)$, because $b \notin B'$ and $B'$ is the set of the $\ell$ highest-degree vertices in $B$.  
So the number of vertices which {\bf cannot} play the role of $\sigma(b)$ is at most $\frac{m}{\ell+1} \cdot \max_{v \in \sigma'(A')}{d_G(v)} + |V(F)| - 1 < N/2 \leq |W|$, where the first inequality uses Item 1. Therefore, there is a suitable choice for $\sigma(b) \in W$.

Suppose now that we have embedded $B \setminus B'$, and let us embed the vertices of $A \setminus A'$ (here we no longer insist that vertices are embedded into $W$). Let $a \in A \setminus A'$. We need to show that there is $\sigma(a) \in V(G)$ such that $\sigma(a)$ is not adjacent in $G$ to $\sigma(b)$ for any $b \in B$ with $(a,b) \in E(F)$, and such that $\sigma(a)$ is different from all previously embedded vertices. As above, we have $d_F(a) \leq m/(k+1)$. For each $b \in B \setminus B'$, we have $d_G(\sigma(b)) \leq 2d$ because $\sigma(b) \in W$. Therefore, the number of vertices which {\bf cannot} play the role of $\sigma(a)$ is at most
$
\frac{m}{k+1} \cdot \max\left( \max_{v \in \sigma'(B')}d_G(v), \; 2d \right) + |V(F)| - 1 < N,$
using Item 2. So there is a valid choice for $\sigma(a) \in V(G)$. 
\end{proof} 

Finally, we will need the following well-known result on the independence number of graphs with few triangles, see e.g. \cite[Lemma 12.16]{Bollobas_RandomGraphs}. 
\begin{lemma}\label{lem:AKS}
Let $G$ be a graph with $N$ vertices, average degree $d$, and at most $T$ triangles. Then $G$ contains an independent set of size at least $0.1 \frac{N}{d} \cdot \left( \log d - \frac{1}{2}\log(T/N) \right)$.
\end{lemma}

\section{Proof of Theorem \ref{thm:vs bipartite}}\label{sec:vs bipartite}
Let us first sketch the proof of Theorem \ref{thm:vs bipartite} in the case $F = K_{n,n}$. So let $G$ be a graph on $N = Cn^2$ vertices with no copy of $K_4^*$. We need to show that $\overbar{G}$ contains a copy of $K_{n,n}$. First, it is easy to see that by deleting some $N/2$ (say) vertices, we may assume that the minimum degree of $G$ is $\Omega(Cn)$. (Else, $G$ contains an independent set of size $2n$, so $\overbar{G}$ contains a $K_{n,n}$, as required.) Let $S$ be the set of pairs of vertices $(x,y)$ such that $x,y$ have at most two common neighbours. Suppose first that there is $x \in V(G)$ such that $S(x) := \{y : (x,y) \in S\}$ has size at least $3n$. In this case, take disjoint sets $A \subseteq N_G(x), B \subseteq S(x)$, each of size $3n$ (this is possible because $d(x) = \Omega(Cn) \geq 6n$). By the definition of $S$, each vertex in $B$ has at most two common neighbours with $x$, hence it has at most two neighbours in $A$. This allows us to find greedily an $n\times n$ empty bipartite graph between $A,B$, as required. 

So from now on suppose that $|S(x)| \leq 3n$ for each $x \in V(G)$, implying that $|S| \leq 3nN/2$.
Now, taking a vertex $v \in V(G)$ at random, we see that the number of pairs $(x,y) \in S$ contained in $N(v)$ is on average at most $3n$ (each pair from $S$ is counted at most twice when averaging over $v$, by the definition of $S$). So fix $v \in V(G)$ with at most $3n$ pairs $(x,y) \in S$ inside $N(v)$. Recall that $|N(v)| = \Omega(Cn)$. We may assume that the average degree inside $N(v)$ is at least $\Omega(C)$, because otherwise $N(v)$ would contain an independent set of size $2n$, and we would be done. It follows that $G[N(v)]$ contains at least $\sum_{u \in N(v)} \binom{d_{G[N(v)]}(u)}{2} = \Omega(C^3 n)$ paths of length two where we used that $\binom{x}{2}$ is convex. Also, each pair $x,y$ can be the endpoints of at most one such path of length two, because otherwise we get a $C_4$ inside $N(v)$, and hence a $K_4^*$ together with $v$. So in $N(v)$ there are at least $\Omega(C^3 n) > 3n$ pairs $(x,y)$ which are the endpoints of a path of length two. Hence, one of these pairs is not in $S$. Fix such a pair $x,y$, and let $x,z,y$ be a path of length two inside $N(v)$. Since $(x,y) \notin S$, there is an additional neighbour $u \notin \{v,z\}$ of $x,y$. Now $v,x,y,z,u$ form a copy of $K_4^*$. 


Unfortunately, we were not able to adapt the above proof in a clean way to work for every bipartite graph $F$. Finding such a concise proof of Theorem \ref{thm:vs bipartite} would be interesting. Instead, to make the proof work for an arbitrary bipartite $F$, we apply regularization to $G$ via Lemma \ref{lem:regularization}, ensuring that the maximum degree is at most a logarithmic factor away from the average degree. 
This ``almost-regularity'' of $G$ will be useful when applying Lemmas \ref{lem:greedy embedding} and \ref{lem:bipartite greedy embedding} in certain steps of the proof.
To compensate for the (extra) logarithmic factor, we use Lemma \ref{lem:AKS}. The details follow. 

\begin{proof}[Proof of Theorem \ref{thm:vs bipartite}]
Let $F$ be a bipartite graph with sides $A,B$, having $m$ edges and no isolated vertices. Note that $|V(F)| \leq 2m$. Let $G$ be a graph on $N = Cm$ vertices with no copy of $K_4^*$, where $C$ is a large enough constant. Our goal is to show that $\overbar{G}$ contains a copy of $F$.
By Lemma \ref{lem:regularization}, there is an induced subgraph $G'$ of $G$ with $|V(G')| \geq N/6$ and $\Delta(G') \leq d(G') \cdot \log|V(G')|$. With a slight abuse of notation, we will use the notation $G$ for $G'$ and $N$ for $|V(G')|$; so $|V(G)| = N$ and $\Delta(G) \leq d(G) \cdot \log N$.
Put $d := d(G)$. 

Let $S$ be the set of pairs $(x,y) \in \binom{V(G)}{2}$ with $d_G(x,y) \leq 2$. We proceed with several cases.
\paragraph{Case 1:} $|S| \geq 2Nd \log N$. For each $x \in V(G)$, let $S(x)$ be the set of $y \in V(G)$ with $d(y) \leq d(x)$ and $(x,y) \in S$. Then $\sum_{x \in V(G)}{|S(x)|} \geq |S|$. Hence, there is $x$ with $|S(x)| \geq 2d\log N$. Since $\Delta(G) \leq d\log N$, we have 
$|S(x) \setminus N_G(x)| \geq d\log N$. Each $y \in S(x)$ has at most $2$ neighbours in $N_G(x)$, by the definition of $S$. By Lemma \ref{lem:greedy empty bipartite graph} with $r = 2$, $X = N_G(x)$ and $Y = S(x) \setminus N_G(x)$, there exist $X'\subseteq N_G(x)$ and $Y'\subseteq S(x) \setminus N_G(x)$ such that $|X'| \geq \lfloor d_G(x)/3 \rfloor$,
$|Y'| \geq \lfloor |S(x) \setminus N_G(x)|/3 \rfloor \geq \lfloor d\log(N)/3 \rfloor$, and there are no edges in $G$ between $X'$ and $Y'$. 

Let $A' \subseteq A$ be the set of the $k := \min\{|Y'|,|A|\}$ highest-degree vertices in $A$, and let $B' \subseteq B$ be the set of the $\ell := \min\{|X'|,|B|\}$ highest-degree vertices in $B$. Map $A'$ into $Y'$ and $B'$ into $X'$ arbitrarily. This mapping $\sigma'$ is an embedding of $F[A' \cup B']$ into $\overbar{G}$, because there are no edges in $G$ between $X'$ and $Y'$. 
We now verify that Items 1-2 in Lemma \ref{lem:bipartite greedy embedding} hold. 
All vertices in $\sigma'(A') \subseteq Y' \subseteq S(x)$ have degree at most $d_G(x)$ by the definition of $S(x)$. Hence (assuming $B' \neq B$), we have 
$$
\frac{2m}{\ell+1} \cdot \max_{v \in \sigma'(A')}{d_G(v)} \leq \frac{2m}{|X'|+1} \cdot d_G(x) \leq \frac{2m}{d_G(x)/3} \cdot d_G(x) = 6m.
$$
Also, $|V(F)| \leq 2m$.
Therefore, Item 1 in Lemma \ref{lem:bipartite greedy embedding} holds provided that $N \geq 10m$. Next, (assuming $A' \neq A$), we have 
$$
\frac{m}{k+1} \cdot \max \left( \max_{v \in \sigma'(B')}{d_G(v)}, \; 2d \right) \leq 
\frac{m}{|Y'|+1} \cdot 2\Delta(G) \leq \frac{m}{d\log(N)/3} \cdot 2d\log N \leq 6m.
$$
So Item 2 in Lemma \ref{lem:bipartite greedy embedding} holds as well, provided that $N \geq 8m$. Hence, $\overbar{G}$ contains a copy of $F$, as \nolinebreak required. 

\paragraph{Case 2:} $|S| \leq 2Nd \log N$ and $d \geq 576\sqrt{m\log N}$. For each $x \in V(G)$, let $N'(x)$ denote the set of neighbours $y$ of $x$ with $d(y) \leq d(x)$, and let $d'(x) = |N'(x)|$. Then $\sum_{x \in V(G)}{d'(x)} \geq e(G) = dN/2$. Let $t(x)$ be the number of pairs $(y,z) \in S$ such that $y,z \in N'(x)$. We have $\sum_{x \in V(G)}{t(x)} \leq 2|S|$, because each pair in $S$ is counted at most twice in this sum, by the definition of $S$. 
Observe that
\begin{align*}
    \sum_{x \in V(G)}{\Big[ 16\log N \cdot \Big(d'(x) - d(x)/8 - d/8 \Big) - t(x)\Big]} &\geq 
    16\log N \cdot e(G)/2 - \sum_{x \in V(G)}{t(x)} \\&= 4dN\log N - \sum_{x \in V(G)}{t(x)} \geq 4dN\log N - 2|S| \geq 0. 
\end{align*}
Hence, there is $x \in V(G)$ with $16\log N \cdot \Big(d'(x) - d(x)/8 - d/8 \Big) - t(x) \geq 0$. In particular, $d'(x) \geq d(x)/8, d/8$, and the number $t(x)$ of pairs $(y,z) \in S$ with $y,z \in N'(x)$ is at most $16d'(x) \log N$. 
Fix such $x$, let $G_1 := G[N'(x)]$ and let $d_1 := d(G_1)$ be the average degree of $G_1$. 

We claim that $d_1 \leq 6\sqrt{\log N}$. Suppose otherwise. Then, by convexity, the number of paths of length $2$ in $G_1$ is at least $|N'(x)| \cdot \binom{6\sqrt{\log N}}{2} > 16d'(x) \log N \geq t(x)$. 
A pair of vertices from $N'(x)$ can be the endpoints of at most one path of length two in $G_1$, because otherwise $G_1 = G[N'(x)]$ would contain a copy of $C_4$, which together with $x$ would give a copy of $K_4^*$ in $G$, a contradiction. So we see that there are more than $t(x)$ pairs $(y,z) \in \binom{N'(x)}{2}$ which are the endpoints of a path of length $2$ in $G_1$. Hence, there is such a pair $(y,z)$ which does {\bf not} belong to $S$. Let $w$ be the middle vertex of the path of length two between $y$ and $z$ in $G_1$. Since $(y,z) \notin S$, there is a common neighbour $u$ of $y,z$ with $u \neq x,w$. Now $x,y,z,w,u$ span a copy of $K_4^*$, a contradiction. This proves the claim that $d_1 \leq 6\sqrt{\log N}$. 

Let $A \subseteq V(F)$ be the set of $k := \min\{d'(x)/3, |V(F)|\}$ highest-degree vertices in $F$. 
Our goal is to embed $F[A]$ into $\overbar{G}[N'(x)]$, and then use Lemma \ref{lem:greedy embedding} to extend this into an embedding of $F$ into $\overbar{G}$. To embed $F[A]$ into $\overbar{G}[N'(x)]$, we use Corollary \ref{cor:degeneracy embedding} with $F[A]$ in the role of $F$ and $G_1 = G[N'(x)]$ in the role of $G$.
To apply the corollary, we need to verify the condition 
\begin{equation}\label{eq:embedding into G_1}
d_1 \leq \sqrt{\frac{|V(G_1)|^2 - 2|V(G_1)| \cdot |A|}{48 e(F[A])}}.
\end{equation}
By definition, $|A| \le d'(x)/3$. Also, $|V(G_1)| = d'(x)$, $d_1 \leq 6\sqrt{\log N}$ and $e(F[A]) \le m$. So
the RHS of \eqref{eq:embedding into G_1} is at least $\frac{d'(x)}{12\sqrt{m}}$, and hence \eqref{eq:embedding into G_1} holds provided that $d'(x) \geq 72\sqrt{m\log N}$, as $d_1 \leq 6\sqrt{\log N}$. But $d'(x) \geq d/8 \geq 72\sqrt{m\log N}$ by the assumption of Case 2, so \eqref{eq:embedding into G_1} indeed holds. By Corollary \ref{cor:degeneracy embedding}, there is an embedding $\sigma'$ of $F[A]$ into $\overbar{G}[N'(x)]$.

We now use Lemma \ref{lem:greedy embedding} to embed $F$ into $\overbar{G}$. Recall that all vertices in $\sigma'(A) \subseteq N'(x)$ have degree at most $d(x) \leq 8d'(x)$ in $G$ (by the definition of $N'(x)$), and that the average degree of $G$ is $d \leq 8d'(x)$. Also, $k = d'(x)/3$ assuming $A \neq V(F)$, by our choice of $k$. 
Therefore, we can bound the first term on the RHS of \eqref{eq:greedy embedding} as follows:
$\frac{4m}{k+1} \cdot \left( \max_{v \in \sigma'(A)}{d_G(v)}, \; 2d \right) \leq 
\frac{12m}{d'(x)} \cdot 16d'(x) = 192m$. Also, $|V(F)| \leq 2m$. 
So \eqref{eq:greedy embedding} holds for $N \geq 196m$. We conclude that $\overbar{G}$ contains a copy of $F$, as required.

\paragraph{Case 3:} $d \leq 576\sqrt{m\log N} = O(\sqrt{N\log N})$ (recall that $N = Cm$ so $m \leq N$).
Let $W$ be the set of vertices of $G$ of degree at most $2d$. Then $|W| \geq N/2$.
Let $d_0 := d(G[W])$ be the average degree of $G[W]$. Then $d_0 \leq \Delta(G[W]) \leq 2d = O(\sqrt{N\log N})$. 

Let us bound the number of triangles in $G[W]$. We have $\Delta(G[W]) \leq 2d = O(\sqrt{N\log N}) = N^{1/2+o(1)}$.   
For each vertex $x$, the neighbourhood of $x$ in $G$ contains no $C_4$, because otherwise $G$ would contain a copy of $K_4^*$. Hence, $x$ participates in at most 
$\Delta(G[W])^{3/2} \leq N^{3/4+o(1)} \leq d_0^{3/2 + o(1)}$ triangles of $G[W]$. So the overall number of triangles in $G[W]$ is at most $T := |W| \cdot d_0^{3/2 + o(1)}$. 
By Lemma \ref{lem:AKS}, applied to $G[W]$, there is an independent set $I$ in $G[W]$ of size at least 
\begin{equation*}\label{eq:greedy embedding Case 3}
\begin{split} 
|I| &\geq 0.1\frac{|W|}{d_0} \cdot \left( \log d_0  - \frac{1}{2}\log \left(\frac{T}{|W|} \right) \right) = 
0.1\frac{|W|}{d_0} \cdot \left(\frac{1}{4} - o(1)\right) \log d_0 \\&\geq 
0.01\frac{N}{d_0} \log d_0 \geq \Omega\left( \sqrt{N\log N} \right) \geq \Omega(d).
\end{split}
\end{equation*}
Here, the penultimate inequality uses that $d_0 \leq O(\sqrt{N\log N})$ and that the function $x \mapsto \frac{\log x}{x}$ is decreasing (for $x \geq e$), and the last inequality uses that $d \leq O(\sqrt{N\log N})$.

Let $A$ be the set of the $k := \min\{|I|,|V(F)|\}$ highest-degree vertices in $F$. We use Lemma \ref{lem:greedy embedding} to embed $F$ into $\overbar{G}$, starting with an arbitrary embedding of $F[A]$ into $I$. Recall that all vertices in $I \subseteq W$ have degree at most $2d$ in $G$. 
Hence, assuming $A \neq V(F)$, the RHS of condition \eqref{eq:greedy embedding} is at most 
$\frac{4m}{k+1} \cdot 2d + 2|V(F)| \leq \frac{4m}{|I|} \cdot 2d + 4m \leq O(m)$, using that $|I| =\Omega(d)$. 
Therefore, condition \eqref{eq:greedy embedding} holds for $N = Cm$, provided that $C$ is large enough compared to the implied constant in the $O$-notation in the previous sentence. So $\overbar{G}$ contains a copy of $F$. 
This completes the proof.
\end{proof}

\section{Proof of Theorem \ref{thm:K4 subdivisions}}\label{sec:K4 subdivisions}

\begin{proof}
There are three subdivisions of $K_4$ on $6$ vertices, and we denote these by $H_1$, $H_2$, $H_3$; see the figure above. Every subdivision of $K_4$ on more than six vertices is a subdivision of $H_i$ for some $i = 1,2,3$.
\begin{figure}
    \centering
    \includegraphics{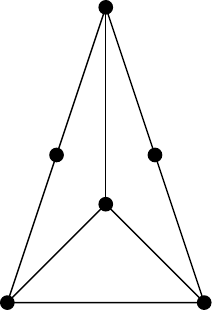}
    \hspace{0.5cm}
    \includegraphics{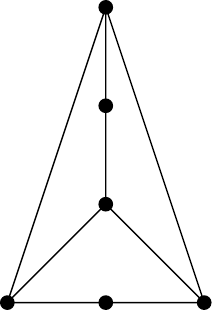}
    \hspace{0.5cm}
    \includegraphics{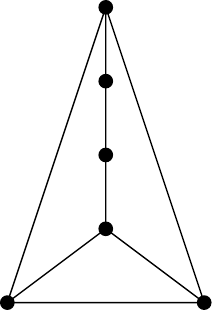}
    \caption{$H_1$, $H_2$, $H_3$ (from left to right)}
\end{figure}
So let $H$ be a subdivision of $H_1$, $H_2$ or $H_3$. Let $h := |V(H)|$. Fix constants $h \ll C_0 \ll C_1 \ll C$, to be chosen implicitly later. 
Let $F$ be a graph with $m$ edges and no isolated vertices. Let $G$ be a graph on $N = Cm$ vertices. We assume that $\overbar{G}$ has no copy of $F$ and our goal is to show that $G$ has a copy of $H$. We begin with some general preparation that will be used in all three cases of $H_1,H_2,H_3$. 
Let $d = d(G)$ be the average degree of $G$. Under the assumption that $\overbar{G}$ has no copy of $F$ (and that $C$ is large enough in terms of $C_1$), we now prove the following:
\begin{claim}\label{claim:greedy embedding}
The following holds:
\begin{enumerate}
    \item $d \geq 0.1N/\sqrt{m} \geq C_1\sqrt{N}$.
    \item There is no independent set $I$ such that $|I| \geq d/C_1$ and all vertices in $I$ have degree at most $C_1 \cdot |I|$.
\end{enumerate}
\end{claim}
\begin{proof}
We show that if Item 1 or 2 does not hold, then $\overbar{G}$ contains a copy of $F$. For Item 1, 
if $d \leq 0.1N/\sqrt{m}$ then $\overbar{G}$ contains a copy of $F$ by Corollary \ref{cor:degeneracy embedding}, using that $|V(F)| \leq 2m$. The inequality $0.1N/\sqrt{m} \geq C_1\sqrt{N}$ in Item 1 holds because $C \gg C_1$, as $N = Cm$. 
Suppose now that Item 2 fails and let $I$ be as in that item. We apply Lemma \ref{lem:greedy embedding}. Take $A$ to be the set of $k = \max\{|I|,|V(F)|\}$ highest-degree vertices in $F$. Map $A$ arbitrarily into $I$; this is an embedding of $F[A]$ into $\overbar{G}$ because $I$ is independent in $G$. Let us verify \eqref{eq:greedy embedding} in Lemma \ref{lem:greedy embedding}. All vertices in $I$ have degree at most $C_1 \cdot |I|$. Also, the average degree $d$ of $G$ satisfies $d \leq C_1 \cdot |I|$.
Hence, assuming $A \neq V(F)$, the RHS in \eqref{eq:greedy embedding} is at most 
$\frac{4m}{|I|+1} \cdot 2C_1 \cdot |I| + 2|V(F)| \leq 8C_1m + 4m$. So \eqref{eq:greedy embedding} holds for $N = Cm$ if $C \gg C_1$. 
\end{proof}

Let $\mathcal{T}$ be the set of triangles in $G$. We run the following process. As long as there is an edge $e$ which is contained in at least one and at most $C_0$ triangles from $\mathcal{T}$, we delete from $\mathcal{T}$ all triangles containing $e$ (we do not make any changes to the graph, only to the collection $\mathcal{T}$). 
We say that $e$ is {\em eliminated} at this step. Note that at a given step of the process, a triangle (still) belongs to $\mathcal{T}$ if and only if none of its edges has been eliminated.
At the end of the process, we are left with a collection of triangles $\mathcal{T}_0$ with the property that if an edge is contained in a triangle from $\mathcal{T}_0$ then it is contained in at least $C_0$ such triangles. 
Let $G_-$ be the subgraph of $G$ consisting of all edges which are eliminated during the process. Let $G^+$ be the subgraph of the remaining edges, i.e. $E(G_+) = E(G) \setminus E(G_-)$. Note that $e \in E(G_+)$ if and only if $e$ is contained in some triangle from $\mathcal{T}_0$. 
Observe that:
\begin{enumerate}
    \item[(a)] For every $e \in E(G_+)$, there are at least $C_0$ triangles in $G_+$ containing $e$. 
     \item[(b)] For every set $B \subseteq E(G_-)$, there is $e \in B$ such that there are at most $C_0$ triangles in $G$ which contain $e$ and only use edges from $B \cup E(G_+)$.
    In particular, the total number of triangles in $G$ containing edges from $G_-$ is at most $C_0 \cdot e(G)$. 
\end{enumerate}
Property (a) holds by the definition of $G_+$ and the above-mentioned property of $\mathcal{T}_0$. 
For property (b), take $e$ to be the earliest eliminated edge among the edges in $B$. Before $e$ is eliminated, all triangles in $G$ which only use edges from $B \cup E(G_+)$ are still in $\mathcal{T}$, because none of the edges in $B \cup E(G_+)$ has been eliminated yet (the edges in $E(G_+)$ are never eliminated). At the moment that $e$ is eliminated, the number of triangles in $\mathcal{T}$ which contain $e$ must be at most $C_0$. 

In what follows, we denote by $N_+(x)$ the neighbourhood of $x$ in $G_+$, and by $N_+(x,y)$ the common neighbourhood of $x,y$ in $G_+$. When writing $N(x)$, we mean the neighbourhood in $G$.  
\begin{claim}\label{claim:subdivision paths}
For every edge $(x,y) \in E(G_+)$, there are paths $P^{x,y}_{1},\dots,P^{x,y}_{h}$ in $G_+$ of length $h$, all starting at $x$ and intersecting only at $x$, such that every vertex on these paths is adjacent to $y$.
\end{claim}
\begin{proof}
Put $N = N_+(y)$, so $x \in N$. 
Observe that by Property (a), the graph $G_+[N]$ has minimum degree at least $C_0$. Hence, one can greedily find the above paths $P^{x,y}_1,\dots,P^{x,y}_h$ inside $N = N_+(y) \subseteq N(y)$, provided that $C_0$ is large enough in terms of $h$.
\end{proof}
Using Claim \ref{claim:subdivision paths}, we can find a copy of $H$ in $G$ by finding copies of $H_1,H_2,H_3$ in $G_+$. 
Recall that $K_4^*$ is the subdivision of $K_4$ where exactly one edge is subdivided once.
\begin{claim}\label{claim:finding subdivision}
Let $K \in \{K_4,K_4^*,H_1,H_2,H_3\}$, and suppose that $H$ is a subdivision of $K$. For $e \in E(K)$, let $\ell_e$ be the length of the subdivision path in $H$ replacing $e$, and let $E_0 = \{e \in E(K) : \ell_e \geq 2\}$. 
If $G$ contains a copy of $K$ in which all edges from $E_0$ are in $G_+$, then $G$ contains a copy of $H$. 
\end{claim}
\begin{proof}
To obtain a copy of $H$, we need to replace the edges $e \in E_0$ with internally-disjoint paths of appropriate lengths. We find these paths one-by-one. Suppose that the path replacing the edge $(x,y) \in E_0$ needs to have length $k \leq h = |V(H)|$. We use the paths $P^{x,y}_i$ given by Claim \ref{claim:subdivision paths}. One of the paths $P^{x,y}_i$, $1 \leq i \leq h$, must be internally disjoint from all vertices embedded so far (because $|V(H)| = h$). Also, since all vertices of this path are adjacent to $y$, we can shorten it to a path that ends in $y$ and has length exactly $k$. This gives a copy of $H$.
\end{proof}

Note that $H_1,H_2,H_3$ are subdivisions of $K_4,K_4^*$, so $H$ is a subdivision of $K_4,K_4^*$ by transitivity. Hence, if $K \in \{K_4,K_4^*\}$, then the condition ``$H$ is a subdivision of $K$'' in Claim \ref{claim:finding subdivision} is always satisfied.

For $v \in V(G)$, let $b(v)$ denote the number of edges $(x,y) \in E(G_-)$ with $x,y \in N(v)$. By Property \nolinebreak (b),
\begin{equation}\label{eq:bad edges}
\sum_{v \in V(G)}{b(v)}  
\leq 3C_0 \cdot e(G).
\end{equation}
Indeed, the sum $\sum_{v \in V(G)}{b(v)}$ counts triangles which contain an edge of $G_-$, and each such triangle is counted at most 3 times. 
\begin{claim}\label{claim:good edges}
$e(G_+) \geq e(G)/2$. 
\end{claim}
\begin{proof}
Suppose by contradiction that $e(G_+) \leq e(G)/2$, and hence $e(G_-) \geq e(G)/2$. For $v \in V(G)$, let $N'(v)$ be the set of vertices $u$ such that $(u,v) \in E(G_-)$ and $d(u) \leq d(v)$, and let $d'(v) = |N'(v)|$. Then $\sum_{v \in V(G)}{d'(v)} \geq e(G_-) \geq e(G)/2$. Observe that
$$
\sum_{v \in V(G)}\left( d'(v) - d(v)/16 - d/16 - \frac{1}{12C_0}\cdot b(v) \right) \geq e(G)/2 - e(G)/8 - e(G)/8 - \frac{1}{12C_0} \cdot \sum_{v \in V(G)}{b(v)} \geq 0,
$$
where the last inequality uses \eqref{eq:bad edges}. 
So there is $v \in V(G)$ such that $d'(v) \geq d(v)/16,d/16$ and $b(v) \leq 12C_0 \cdot d'(v)$. 
In particular, this means that there are at most $12C_0 \cdot d'(v)$ edges of $G_-$ inside $N'(v)$.
So the average degree of $G_-[N'(v)]$ is at most $24C_0$.
Hence, there is $A \subseteq N'(v)$ which is independent in $G_-$ with $|A| \geq d'(v)/(24C_0 + 1)$. Note that $(v,u) \in E(G_-)$ every $u \in A$, by the definition of $N'(v)$. 

We claim that the graph $G[A]$ is $C_0$-degenerate.
Indeed, given any $S \subseteq A$, apply Property (b) to the set of edges $B := \{(v,u) : u \in S\} \subseteq E(G_-)$, and let $e = (v,u) \in B$ be the edge given by Property (b). 
Since all edges in $G[S]$ belong to $G_+$, every edge of the form $(u,w) \in E(G[S])$ forms a triangle $u,v,w$ in which $(v,u),(v,w) \in B$ and $(u,w) \in E(G_+)$. By our choice of $e$, there are at most $C_0$ such triangles, so $d_{G[S]}(u) \leq C_0$, as required. 
It follows that $G[A]$ contains an independent set $I$ of size at least $|I| \geq \frac{|A|}{C_0 + 1} \geq \frac{d'(v)}{(C_0+1)(24C_0+1)} \geq d/C_1$, as $d'(v) \geq d/16$ (provided $C_1 \gg C_0$). Also, all vertices in $I \subseteq N'(v)$ have degree at most $d(v)$ in $G$, by the definition of $N'(v)$, and $d(v) \leq 16d'(v) \leq 16(C_0+1)(24C_0+1) \cdot |I| \leq C_1 \cdot |I|$.
But this contradicts Item 2 of Claim \ref{claim:greedy embedding}.
\end{proof}

\noindent
We now proceed by case analysis over the cases of $H_1,H_2,H_3$.

\paragraph{Case 1: $H$ is a subdivision of $H_1$.}
For $v \in V(G)$, denote by $N^*(v)$ the set of all $u \in V(G)$ such that $(u,v) \in E(G_+)$ and $d(u) \leq d(v)$, and let $d^*(v) = |N^*(v)|$. Then $\sum_{v \in V(G)}{d^*(v)} \geq |e(G_+)| \geq e(G)/2$, by Claim \ref{claim:good edges}.
Let $P(v)$ denote the set of pairs $(u,w) \in \binom{N^*(v)}{2}$ such that $|N_+(u,w)| \leq C_0$ (i.e., $u,w$ have at most $C_0$ common neighbours in $G_+$). 
Let $p(v) = |P(v)|$. Observe that
\begin{equation}\label{eq:p(v)}
\sum_{v \in V(G)}{p(v)} \leq C_0 \binom{N}{2},
\end{equation}
because each pair $(u,w)$ is counted at most $C_0$ times in the sum on the LHS of \eqref{eq:p(v)}, by the definition of the sets $P(v)$.
Now, note that
\begin{align*}
    &\sum_{v \in V(G)}{\left( d^*(v) - d(v)/16 - d/16 - \frac{1}{24C_0} \cdot b(v) - \frac{e(G)}{4C_0N^2} \cdot p(v) \right)} 
    = \\& 
    \sum_{v \in V(G)}{d^*(v)} - e(G)/4 - \frac{1}{24C_0} \cdot \sum_{v \in V(G)}{b(v)} - \frac{e(G)}{4C_0N^2} \cdot \sum_{v \in V(G)}{p(v)} 
    \geq 
    \\ &
    e(G)/2 - e(G)/4 - e(G)/8 - e(G)/8 = 0.
\end{align*}
Here we used \eqref{eq:bad edges} and \eqref{eq:p(v)}. 
So there is $v \in V(G)$ such that $d^*(v) \geq d(v)/16,d/16$, $b(v) \leq 24C_0 \cdot \nolinebreak d^*(v)$ \nolinebreak and 
\begin{equation}\label{eq:bad pairs in N(v)}
p(v) \leq \frac{4C_0N^2}{e(G)} \cdot d^*(v) = \frac{8C_0N}{d} \cdot d^*(v) \leq \frac{d^*(v) \cdot (d^*(v) - 3)}{8},
\end{equation}
where the last inequality holds if $C_1 \gg C_0$ because $d^*(v) \geq d/16 \geq C_1\sqrt{N}/16$, using Item 1 of Claim \nolinebreak \ref{claim:greedy embedding}. 

Let $A \subseteq N^*(v)$ be the set of all $u \in N^*(v)$ which participate in at most $\frac{d^*(v)-3}{2}$ of the pairs in $P(v)$. 
Then $p(v) \geq \frac{1}{2}(d^*(v) - |A|) \cdot \frac{d^*(v)-3}{2}$, so $|A| \geq d^*(v)/2$ by 
\eqref{eq:bad pairs in N(v)}. 
We claim that
$e(G[A]) > 48C_0 \cdot |A|$. Indeed, otherwise $G[A]$ would contain an independent set $I$ of size $|I| \geq \frac{|A|}{96C_0+1} \geq \frac{d^*(v)}{2(96C_0+1)} \geq \frac{d}{C_1}$, as $d^*(v) \geq d/16$ and $C_1 \gg C_0$. Also, $I \subseteq N^*(v)$ and all vertices in $N^*(v)$ have degree at most $d(v) \leq 16d^*(v) \leq C_1 \cdot |I|$ in $G$.
This would contradict Item 2 of Claim \ref{claim:greedy embedding}. 
So indeed $e(G[A]) > 48C_0 \cdot |A|$. On the other hand, $G[A]$ contains at most $b(v)$ edges of $G_-$ (by the definition of $b(v)$). 
Hence, $G[A]$ contains at least $e(G[A]) - b(v) > 48C_0 \cdot |A| - b(v) \geq 48C_0 \cdot d^*(v)/2 - b(v) \geq 0$ edges of $G_+$. 

Fix an edge $(u_1,u_2) \in E(G_+)$ with $u_1,u_2 \in A$. 
By the definition of $A$, each $u_i$ participates in at most $\frac{d^*(v)-3}{2}$ of the pairs in $P(v)$. Hence, there is 
$w \in N^*(v)$ different from $u_1, u_2$ such that $(u_1,w),(u_2,w) \notin P(v)$. By the definition of $P(v)$, this means that $u_i,w$ have at least $C_0$ common neighbours in $G_+$ for $i = 1,2$. Let $z_i$ be a common neighbour of $u_i,w$ in $G_+$, such that $v,u_1,u_2,w,z_1,z_2$ are all distinct. Then these six vertices form a copy of $H_1$ in $G_+$. 
Now, by Claim \ref{claim:finding subdivision}, $G$ contains a copy of $H$. This concludes the proof in Case 1.
\paragraph{Case 2: $H$ is a subdivision of $H_2$.} 
We have $\sum_v d_{+}(v) = 2e(G_+) \geq e(G)$, where the inequality is by Claim \ref{claim:good edges}. Sample two distinct vertices $v_1,v_2 \in V(G)$ uniformly at random and let $A = N_+(v_1,v_2)$ be the common neighbourhood of $v_1,v_2$ in $G_+$. For each $u \in V(G)$, the probability that $u \in A$ is 
$\binom{d_+(u)}{2}/{\binom{N}{2}}$.
By Jensen's inequality, 
\begin{align*}
\mathbb{E}[|A|] &= \frac{1}{{\binom{N}{2}}} \cdot \sum_{u}{\binom{d_+(u)}{2}} \geq 
\frac{N}{\binom{N}{2}} \cdot \binom{\frac{1}{N} \cdot \sum_{u}d_+(u)}{2} \geq 
\frac{N}{\binom{N}{2}} \cdot \binom{e(G)/N}{2} \\ &\geq 
\frac{2}{N} \cdot \frac{(e(G)/N)^2}{4} = \frac{e(G)^2}{2N^3} \geq C_0^2,
\end{align*}
where the last inequality holds because $e(G) = dN/2$, $d \geq C_1 \sqrt{N}$ by Item 1 of Claim \ref{claim:greedy embedding}, and $C_1 \gg C_0$.

Let $P$ be the set of pairs of vertices $u_1,u_2 \in A$ such that $|N_+(u_1,u_2)| \leq C_0$. For a given pair $u_1,u_2 \in V(G)$ with $|N_+(u_1,u_2)| \leq C_0$, the probability that $u_1,u_2 \in A$ is at most $\binom{C_0}{2}/\binom{N}{2}$. Hence, $\mathbb{E}[|P|] \leq \binom{C_0}{2}$. By linearity of expectation, $\mathbb{E}[|A| - |P|] \geq C_0^2/2 \geq 3$. Hence, there is a choice of $v_1,v_2$ for which $|A| - |P| \geq 3$. By removing one vertex from each pair in $P$, we obtain a subset $A' \subseteq A$, $|A'| \geq 3$, such that no pair of vertices in $A'$ belongs to $P$. Fix distinct $u_1,u_2,u_3 \in A'$. Since $(u_1,u_2) \notin P$, there are more than $C_0 \geq 4$ common neighbours of $u_1,u_2$ in $G_+$. Hence, there is $w \notin \{v_1,v_2,u_3\}$ which is a common neighbour of $u_1,u_2$ in $G_+$. Now, $v_1,v_2,u_1,u_2,u_3,w$ form a copy of $H_2$ in $G_+$.
So by Claim \ref{claim:finding subdivision}, $G$ contains a copy of $H$. This concludes the proof in Case 2.
\paragraph{Case 3: $H$ is a subdivision of $H_3$.}
If $H$ is obtained from $K_4$ by subdividing at least two edges (some number of times), then $H$ is a subdivision of $H_1$ or $H_2$; so such $H$ are already covered by Cases 1-2. Hence, we may assume that $H$ is obtained from $K_4$ by subdividing exactly one edge (some number of times). 
It follows, by Claim \ref{claim:finding subdivision}, that if $G$ contains a copy of $K_4$ in which at least one edge is in $G_+$, then $G$ contains a copy of $H$. 

Note that $H_3$ is obtained from $K_4$ by subdividing one edge twice. So $H_3$ has three {\em subdivision edges} (i.e., the edges of the path replacing the subdivided edge of $K_4$). Denote these edges by $e_1,e_2,e_3$. 
We can treat $H$ as a subdivision of $H_3$ in which only one of $e_1,e_2,e_3$ is subdivided. Indeed, if $e_i$ is replaced in $H$ by a path of length $\ell_i$, then we can replace just one of $e_1,e_2,e_3$ with a path of length $\ell_1+\ell_2+\ell_3$, and keep the other two edges. This means that if $G$ contains a copy of $H_3$ in which all edges except at most two of the edges $e_1,e_2,e_3$ are in $G_+$, then $G$ contains a copy of $H$ (again using Claim \ref{claim:finding subdivision}).

Recall that a diamond is the graph consisting of two triangles sharing an edge. 
A diamond has two vertices of degree $2$ and two vertices of degree $3$; the vertices of degree $2$ will be called the {\em tips} of the diamond, and the edge connecting the two vertices of degree $3$ will be called the {\em middle edge} of the diamond.

By Claim \ref{claim:good edges}, $e(G_+) \geq e(G)/2$. By Property (a) above, for every $e \in E(G_+)$ there are at least $C_0$ triangles in $G_+$ containing $e$. It follows that $G_+$ contains at least $e(G)/2 \cdot \binom{C_0}{2} \geq 4e(G)$ diamonds. 
For $v \in V(G)$, let $t(v)$ denote the number of diamonds $D$ in $G_+$ such that $v$ is a tip of $D$, and the other tip $u$ of $D$ satisfies $d(u) \leq d(v)$. Then $\sum_v {t(v)} \geq 4e(G)$. It follows that 
$$
\sum_{v \in V(G)}{\left( t(v) - d(v) - d \right)} = \sum_{v \in V(G)}{t(v)} - 4e(G) \geq 0.
$$
Hence, there is a vertex $v$ satisfying $t(v) \geq d(v), d$. Fix such a vertex $v$. By definition, there are diamonds $D_1,\dots,D_r$ in $G_+$, $r = t(v)$, such that $v$ is a tip of $D_i$, and the other tip $u_i$ of $D_i$ satisfies $d(u_i) \leq d(v)$ (for $i = 1,\dots,r$). 
Let $e_i$ be the middle edge of $D_i$. Suppose first that there is $1 \leq i \leq r$ such that $(v,u_i) \in E(G)$. Then the vertices of $D_i$ form a $K_4$ in which all edges except possibly $(v,u_i)$ are in $G_+$. As we saw above, this implies that $G$ contains a copy of $H$, completing the proof. So from now on we may assume that $v$ is not connected to any $u_i$. It follows that $\{u_1,\dots,u_r\} \cap \bigcup_{i = 1}^{r}{e_i} = \emptyset$, since $v$ is connected with to all vertices of $\bigcup_{i = 1}^{r}{e_i}$. Next, suppose that there are $1 \leq i < j \leq r$ such that $u_i = u_j$. 
Since $D_i \neq D_j$, there is $x_j \in e_j$ such that $x_j \notin V(D_i)$. Observe that $V(D_i) \cup \{x_j\}$ spans a copy of $K_4^*$ in $G_+$. By Claim \ref{claim:finding subdivision}, this implies that $G$ contains a copy of $H$, completing the proof. 
So from now on we may assume that $u_1,\dots,u_r$ are pairwise distinct. 

We claim that $U := \{u_1,\dots,u_r\}$ is not an independent set of $G$. Indeed, observe that $|U| = r = t(v) \geq d$ and all vertices in $U$ have degree at most $d(v) \leq t(v) = |U|$. So if $U$ were independent then we would get a contradiction to Item 2 of Claim \ref{claim:greedy embedding}.
Let us then fix $1 \leq i < j \leq r$ such that $(u_i,u_j) \in E(G)$. If $e_i = e_j$ then $e_i \cup \{u_i,u_j\}$ spans a copy of $K_4$ in which all edges except possibly $(u_i,u_j)$ are in $G_+$. Again, this implies that $G$ contains a copy of $H$. Suppose finally that $e_i \neq e_j$. Let $x_j \in e_j \setminus V(D_i)$. 
It is easy to check that $V(D_i) \cup \{x_j,u_j\}$ contains a copy of $H_3$ in which all edges except possibly $(u_i,u_j)$ are in $G_+$, and $(u_i,u_j)$ plays the role of one of the subdivision edges $e_1,e_2,e_3$ 
(the edges playing the roles of $e_1,e_2,e_3$ in this copy are $(v,x_j),(x_j,u_j),(u_j,u_i)$). 
As explained above, this implies that $G$ contains a copy of $H$. This completes the proof of Case 3 and hence the theorem. 
\end{proof}

\section{On Conjecture \ref{conj:e-v}}\label{sec:e-v}
This section is broken into several parts. First, we prove Proposition \ref{prop:tw}, which then allows us to prove Theorem \ref{thm:v-e, k=3}. Next, we show that Conjecture \ref{conj:e-v} holds if $K_n$ is replaced by $K_{n,n}$ (see Proposition \ref{prop:e-v, complete bipartite}). To that end, we also study the Ramsey number $R(H,K_{n,n})$ for graphs $H$ of bounded maximum degree and for degenerate graphs (see Section \ref{subsec:H vs complete bipartite}). 

\subsection{Proof of Proposition \ref{prop:tw}}\label{subsec:tw}

Here we prove Proposition \ref{prop:tw}, which bounds the Ramsey number $R(H,K_n)$ in terms of the treewidth of $H$. We refer to \cite{HW} for the basic definitions related to treewidth. 
We will need the following lemma. For a graph $G$, let $\#K_r(G)$ denote the number of $r$-cliques in $G$. 

\begin{lemma}\label{lem:r-cliques}
For any $r \geq 1$ there is $C_r > 0$ such that the following holds. If $G$ is a graph on $N \geq C_r n^r$ vertices with no independent set of size $n$, then $\#K_{r+1}(G) \geq \frac{N}{C_rn^r} \cdot \#K_r(G)$.
\end{lemma}
\begin{proof}
We will show that one can take $C_1 = 4$ and $C_r = 8(r+1) \cdot C_{r-1}$ for $r \geq 2$. (We make no effort of optimising the value of $C_r$.)
The proof is by induction on $r$. Suppose first $r = 1$. Let $d$ be the average degree of $G$. 
We have $n > \alpha(G) \geq \frac{N}{d + 1}$, and hence $\frac{2e(G)}{N} = d > \frac{N}{n} - 1 \geq \frac{N}{2n}$. It follows that $e(G) \geq \frac{N}{4n} \cdot N$, as required. 

Let now $r \geq 2$, and let $G$ be as in the statement of the lemma. By the induction hypothesis, we know that 
$\#K_r(G) \geq \frac{N}{C_{r-1} n^{r-1}} \cdot \#K_{r-1}(G)$. Let $\mathcal{C}$ be the set of all $(r-1)$-cliques $X$ in $G$ such that the number of $r$-cliques containing $X$ is at least $\frac{r}{2} \cdot \frac{N}{C_{r-1}n^{r-1}}$. 
Observe that the number of $r$-cliques which do not contain any $(r-1)$-clique from $\mathcal{C}$ is at most 
$$\frac{1}{r} \cdot \#K_{r-1}(G) \cdot \frac{r}{2} \cdot \frac{N}{C_{r-1}n^{r-1}} \leq \frac{1}{2} \cdot \#K_r(G).$$ 
Hence there are at least $\frac{1}{2} \cdot \#K_r(G)$ $r$-cliques which contain some $(r-1)$-clique from $\mathcal{C}$. 

For each $X \in \mathcal{C}$, let $N(X)$ be the set of vertices $y$ such that $X \cup \{y\}$ is an $r$-clique. By definition, 
$$|N(X)| \geq \frac{r}{2} \cdot \frac{N}{C_{r-1}n^{r-1}} \geq 4n.$$ By the case $r=1$ of the lemma, applied to the graph $G[N(X)]$, we have $e(N(X)) \geq \frac{|N(X)|^2}{4n}$. 
Summing over all $X \in \mathcal{C}$, we see that
\begin{align*}
\binom{r+1}{2} \cdot \#K_{r+1}(G) &\geq \sum_{X \in \mathcal{C}}{e(N(X))} \geq 
\sum_{X \in \mathcal{C}}{\frac{|N(X)|^2}{4n}} 
\geq \frac{1}{4n} \cdot \frac{r}{2} \cdot \frac{N}{C_{r-1}n^{r-1}} \cdot \sum_{X \in \mathcal{C}}{|N(X)|} \\ &\geq
\frac{r}{8} \cdot \frac{N}{C_{r-1}n^r} \cdot \frac{\#K_r(G)}{2} = \binom{r+1}{2} \cdot \frac{N}{C_rn^r} \cdot \#K_r(G).
\end{align*}
\end{proof}

In the proof of Proposition \ref{prop:tw}, it is convenient to work with a tree-decomposition of $H$ in which all bags have size $\text{tw}(H) + 1$, and every two adjacent bags intersect in $\text{tw}(H)$ vertices. It is well-known that such a tree decomposition always exists, see e.g. \cite[Lemma 2]{HW}. 
\begin{lemma}[\cite{HW}]\label{lem:k-tree}
Let $H$ be a graph with $r := \textrm{tw}(H)$. Then there is a tree-decomposition of $H$ in which every bag has size $r+1$ and every two adjacent bags intersect in $r$ vertices. 
\end{lemma}
\noindent
We are now ready to prove Proposition \ref{prop:tw}, which we restate here for convenience. 
\restateproptw*
\begin{proof}
Put $r := \textrm{tw}(H)$. Take a tree-decomposition of $H$ with the properties guaranteed in Lemma \ref{lem:k-tree}; let $T$ be the corresponding tree, and let $B(t)$ be the bag corresponding to $t \in V(T)$. Let $H'$ be the graph obtained by making each bag $B(t)$ a clique; so $V(H') = V(H)$ and $H'$ contains $H$ as a subgraph.
We will show that $R(H',K_n) = O(n^r)$. 
Let $G$ be a graph on $N = Cn^r$ vertices with no independent set of size $n$. By Lemma \ref{lem:r-cliques}, if $C$ is large enough then $\#K_{r+1}(G) > (v(H) - \nolinebreak r - \nolinebreak 1) \cdot \nolinebreak \#K_r(G)$. We now run the following process with sets $\mathcal{C}_{r+1}, \mathcal{C}_r$. Initialize $\mathcal{C}_{r+1}$ to be the set of all $(r+1)$-cliques in $G$, and $\mathcal{C}_r$ to be the set of all $r$-cliques in $G$. As long as there is $X \in \mathcal{C}_r$ such that the number of $Y \in \mathcal{C}_{r+1}$ containing $X$ is at most $v(H) - r - 1$, delete $X$ from $\mathcal{C}_r$ and delete all such $Y$ from $\mathcal{C}_{r+1}$. 
The number of elements of $\mathcal{C}_{r+1}$ deleted throughout the process is at most $\#K_r(G) \cdot (v(H) - r - 1) < \#K_{r+1}(G)$. Hence, the terminal set $\mathcal{C}_{r+1}$ is non-empty. By construction, this set has the property that for every $Y \in \mathcal{C}_{r+1}$ and every $X \subseteq Y$, $|X| = r$, there are at least $v(H) - r$ sets $Y' \in \mathcal{C}_{r+1}$ which contain $X$.

Fix an order $t_1,\dots,t_m$ of $V(T)$ such that $t_i$ has exactly one neighbour in $\{t_1,\dots,t_{i-1}\}$. We now embed $B(t_1),\dots,B(t_m)$ one-by-one, such that the image of each $B(t_i)$ equals some $Y_i \in \mathcal{C}_{r+1}$. Fix an arbitrary $Y_1 \in \mathcal{C}_{r+1}$ and embed $B(t_1)$ onto $Y_1$. For $i \geq 2$, suppose that we already embedded $B(t_1),\dots,B(t_{i-1})$. There is a unique $1 \leq j \leq i-1$ such that $t_j$ is a neighbour of $t_i$. By the definition of tree-decomposition, we have $B(t_i) \cap (B(t_1) \cup \dots \cup B(t_{i-1})) = B(t_i) \cap B(t_j)$. 
By our choice of the tree-decomposition and of $H'$, the intersection $B(t_i) \cap B(t_j)$ is an $r$-clique. Hence, there is a unique vertex $v \in B(t_i) \setminus B(t_j)$. 
Let $X \subseteq Y_j$ be the $r$-clique playing the role of $B(t_i) \cap B(t_j)$. 
There are at least $v(H) - r$ different $(r+1)$-cliques $Y \in \mathcal{C}_{r+1}$ containing $X$; hence for one of these $Y$, the (unique) vertex in $Y \setminus X$ is ``new'', i.e. not contained in $Y_1 \cup \dots \cup Y_{i-1}$. We can now embed $B(t_i)$ onto $Y_i := Y$, mapping $v$ to this new vertex. This completes the proof.  
\end{proof}

\subsection{Proof of Theorem \ref{thm:v-e, k=3}}
The following lemma, appearing in \cite{EFRS}, allows us to assume that $H$ is $2$-connected. For completeness, we include a proof. 
\begin{lemma}[\cite{EFRS}]\label{lem:2-connected}
Let $H$ be a graph obtained from graphs $H_1,H_2$ by gluing them together along a vertex. Then for every graph $F$, $R(H,F) = O\left( R(H_1,F) + R(H_2,F) \right)$.
\end{lemma}
\begin{proof}
    Put $M = \max\{ R(H_1,F), R(H_2,F) \}$, and let $G$ be a graph on $N = (|V(H_1)| + 1) \cdot M$ vertices such that $\overbar{G}$ contains no copy of $F$. We can then find in $G$ vertex-disjoint copies $H_1^{(1)},\dots,H_1^{(M)}$ of $H_1$. Let $v$ be the unique common vertex of $H_1$ and $H_2$. For $i = 1,\dots,M$, let $v^{(i)}$ be the vertex of $H_1^{(i)}$ playing the role of $v$. The subgraph of $G$ induced on $\{v^{(1)},\dots,v^{(M)}\}$ contains a copy $H'_2$ of $H_2$. Let $1 \leq i \leq M$ such that $v^{(i)}$ plays the role of $v$ in $H'_2$. Then $H_1^{(i)} \cup H'_2$ form a copy of $H$.  
\end{proof}
\begin{corollary}[\cite{EFRS}]\label{cor:2-connected}
Let $H$ be a graph with biconnected components $H_1,\dots,H_m$. Then for every graph $F$, $R(H,F) = O(R(H_1,F) + \dots + R(H_m,F))$.
\end{corollary}
\begin{proof}[Proof of Theorem \ref{thm:v-e, k=3}]
    By Corollary \ref{cor:2-connected}, we may assume that $H$ is $2$-connected. Indeed, let $H_1,\dots,H_m$ be the biconnected components of $H$. By Corollary \ref{cor:2-connected}, it is enough to prove that $R(H_i,K_n) = O(n^3)$ for every $i = 1,\dots,m$. Also, $e(H_i) - v(H_i) \leq e(H) - v(H) \leq 4$, because $H$ is connected. Hence, from now on we assume that $H$ is $2$-connected. 
    
    Suppose first that $v(H) \leq 5$. In this case we show that $\text{tw}(H) \leq 3$, which would imply that $R(H,K_n) = O(n^3)$ by Proposition \ref{prop:tw}. If $v(H) = 4$ then $\text{tw}(H) \leq \text{tw}(K_4) = 3$. If $v(H) = 5$ then $e(H) \leq v(H) + 4 = 9$, so $H$ is contained in $K_5 - e$. Note that $K_5 - e$ is obtained by gluing two copies of $K_4$ along a triangle. It is now easy to see that $\text{tw}(K_5 - e) \leq 3$, as required. 
    
    For the rest of the proof, suppose that $v(H) \geq 6$. If $\Delta(H) \leq 2$ then $H$ is a cycle or a path, and it is well-known that in this case $R(H,K_n) = O(n^2)$ (for example, this follows from the case $k=2$ of Conjecture \ref{conj:e-v}, which was proved in \cite{EFRS}). 
    Let $v \in V(H)$ be a vertex of maximum degree, $d_H(v) \geq 3$. Let $H' = H - v$. Note that $H'$ is connected because $H$ is $2$-connected. Also,
    $e(H') - v(H') = (e(H) - d_H(v)) - (v(H) - 1) = e(H) - v(H) - d_H(v) + 1 \leq 5 - d_H(v)$.
    
    We claim that 
    $R(H',K_n) = O(n^2)$. If $d_H(v) \geq 4$ then $e(H') - v(H') \leq 1$, so $R(H',K_n) = O(n^2)$ follows from the case $k = 2$ of Conjecture \ref{conj:e-v}, which was proved in \cite{EFRS}. Suppose now that $d_H(v) = 3$, so $e(H') - v(H') \leq 2$. 
    If $\text{tw}(H') \leq 2$ then $R(H',K_n) = O(n^2)$ by Proposition \ref{prop:tw}, so suppose that $\text{tw}(H') > 2$. It is known (see e.g. \cite{Bodlaender}) that a graph has treewidth larger than $2$ if and only if it contains a subdivision of $K_4$. 
    So $H'$ contains a subdivision $S$ of $K_4$. Observe that $e(S) - v(S) = 2$ (this holds for every subdivision of $K_4$). This implies that $e(H') - v(H') = 2$, and that every 2-connected component of $H'$ other than $S$ is a singleton (this can also be stated as saying that the $2$-core of $H'$ is $S$). It now follows from Corollary \ref{cor:2-connected} that $R(H',K_n) = O(R(S,K_n))$. Now, if $v(S) \geq 6$ then by Theorem \ref{thm:K4 subdivisions} we have $R(S,K_n) = O(n^2)$ and hence $R(H',K_n) = O(n^2)$. So suppose that $v(S) \leq 5$. If $v(S) = 4$, namely $S \cong K_4$, then, since $H$ is connected and has maximum degree 3, it holds that $H \cong K_4$, in contradiction to $v(H) \geq 6$. If $v(S) = 5$ then $S \cong K_4^*$. Note that $K_4^*$ has four vertices of degree 3 and one vertex of degree 2. Let $u$ be this vertex of degree $2$ in $S$. We have $V(H) \setminus V(S) \neq \emptyset$ because $v(H) \geq 6$. Also, there are no edges in $H$ between $V(S) \setminus \{u\}$ and $V(H) \setminus V(S)$, because the vertices in $V(S) \setminus\{u\}$ have degree 3 in $S$ and $\Delta(H) = 3$. So $u$ is a cut vertex of $H$, in contradiction to the fact that $H$ is $2$-connected. This proves that $R(H',K_n) = O(n^2)$.
    
    Now let $G$ be a graph on $N = Cn^3$ vertices with no independent set of size $n$. There exists $x \in V(G)$ with $d(x) \geq Cn^2 - 1$. By choosing $C$ large enough, we can make sure that $d(x) \geq R(H',K_n)$. Then, $G[N(x)]$ contains a copy of $H'$. Together with $x$, we get a copy of $H$, as required.
\end{proof}

\subsection{On the Ramsey number $R(H,K_{n,n})$}\label{subsec:H vs complete bipartite}
We begin by proving upper bounds on $R(H,K_{n,n})$ for graphs $H$ with $\Delta(H) = r$ and for $r$-degenerate $H$. Both of our results follow from Lemma \ref{lem:strongly degenerate} below. First, we need the following definition.  
\begin{definition}
    We say that a graph $H$ is \emph{$r$-strongly-degenerate} if there exists an ordering $v_1, \dots, v_h$ of its vertices such that for all $i \in [h]$ one of the following holds:
    \begin{enumerate}[label=\alph*)]
        \item $|N(v_i) \cap \{v_1, \dots, v_{i-1}\}| \le r-1,$ or
        \item $d(v_i) \le r.$
    \end{enumerate}
\end{definition}

Equivalently, a graph $H$ is $r$-strongly-degenerate if its subgraph induced by the set of vertices with degree larger than $r$ is $(r-1)$-degenerate.

\begin{lemma}\label{lem:strongly degenerate}
    For any $r$-strongly-degenerate graph $H$ on $h$ vertices, $R(H, K_{n,n}) \le h^2 n^r.$
\end{lemma}
\begin{proof}
    Let $u_1, \dots, u_h$ be an ordering of the vertices of $H$ certifying that $H$ is $r$-strongly-degenerate. We denote $d_i(j) = |N_H(u_j) \cap \{ u_1, \dots, u_{i-1}\}|.$ Consider an arbitrary graph $G$ on $N = h^2 n^r$ vertices. We show how to find either a copy of $H$ or a copy of $\overbar{K_{n,n}}.$ Split the vertex-set into $h$ parts $V_1, \dots, V_h$ each of size $hn^r.$ We will try to find an embedding $\phi \colon H \rightarrow V$ such that $\phi(u_j) \in V_j, $ for all $j \in [h].$ For this purpose, we will maintain sets $A_{i, j}$ into which we can embed the vertices, starting with $A_{1, j} = V_j, \, j \in [h].$ We will maintain the following. For any $1 \le i \le j \le h,$ 
    \begin{equation} \label{eq:set_sizes}
        |A_{i,j}| \ge h \cdot n^{r - d_i(j)},
    \end{equation}
    which is trivially satisfied for $i = 1.$ 
    
    Next we describe how to embed $H.$ Suppose we have embedded $u_1, \dots, u_{i-1}$ and we wish to embed $u_i.$ First suppose there exists a vertex $v \in A_{i, i}$ satisfying $|N(v) \cap A_{i, j}| \ge h n^{r - d_{i+1}(j)},$ for all $j > i$ such that $(u_i,u_j) \in E(H).$ Then, we set $\phi(u_i) = v$ and update the sets as follows:
    \[ A_{i+1, j} = 
    \begin{cases}
        A_{i, j} \cap N(v) &\text{if } (u_i,u_j) \in E(H),\\
        A_{i, j} &\text{otherwise.}
    \end{cases} \]
    It directly follows that \eqref{eq:set_sizes} is still satisfied.

    If we can embed all $h$ vertices in this manner, we obtain a copy of $H.$ Hence, for some $i,$ there is no vertex $v \in A_{i, i}$ satisfying $|N(v) \cap A_{i, j}| \ge n^{r - d_{i+1}(j)},$ for all $j > i$ such that $(u_i,u_j) \in E(H).$ Since $A_{i, i} \neq \emptyset$ (by \eqref{eq:set_sizes}), $u_i$ has a neighbour $u_j$ in $H$ with $j > i.$ By definition of an $r$-strongly-degenerate graph, it follows that $d_i(i) \le r-1$ so $|A_{i,i}| \ge hn$ by \eqref{eq:set_sizes}. By the pigeonhole principle, there is an index $k > i$ such that for at least $n$ vertices $v \in A_{i, i}$ we have $|N(v) \cap A_{i, k}| < hn^{r - d_{i+1}(k)}.$ Let $S \subseteq A_{i,i}$ be a set of $n$ such vertices and let $R = A_{i,k} \setminus \bigcup_{v \in S} N(v).$ By assumption $|N(v) \cap A_{i,k}| \le hn^{r-d_{i+1}(k)} - 1$ for every $v \in S$. Note that $d_{i+1}(k) = 1 + d_i(k)$ since $(u_i,u_k) \in E(H).$ Therefore,
    \[ |R| \ge |A_{i,k}| - n \cdot (hn^{r - d_{i+1}(k)}-1) \ge hn^{r - d_i(k)} - n \cdot (hn^{r-d_{i+1}(k)} - 1) =  n.\]
    By construction, $G[R, S]$ is empty which completes the proof.
\end{proof}

Note that every graph with maximum degree $r$ is $r$-strongly-degenerate, and every $r$-degenerate graph is $(r+1)$-strongly-degenerate. Hence, we have the following corollaries.

\begin{corollary}\label{cor:max degree}
    For any graph $H,$ $R(H, K_{n,n}) = O(n^{\Delta(H)}).$
\end{corollary}

\begin{corollary}\label{cor:degenerate}
    For any $r$-degenerate graph $H,$ $R(H, K_{n,n}) = O(n^{r+1}).$
\end{corollary}

\noindent
Finally, we show that Conjecture \ref{conj:e-v} holds if $K_n$ is replaced with $K_{n,n}$. 

\begin{proposition}\label{prop:e-v, complete bipartite}
Let $k \geq 1$. For every connected graph $H$ with $e(H) - v(H) \leq \binom{k+1}{2} - 2$ it holds that $R(H,K_{n,n}) = O(n^k)$. 
\end{proposition}
\begin{proof}
    The proof is by induction on $k$.
    As in the proof of Theorem \ref{thm:v-e, k=3}, we may assume that $H$ is $2$-connected due to Corollary \ref{cor:2-connected}. If $\Delta(H) \leq k$ then we are done by Corollary \ref{cor:max degree}. Else, let $v \in V(H)$ with $d(v) \geq k+1$, and let $H' = H - v$. Then $e(H') - v(H') \leq e(H) - (k+1) - v(H) + 1 \leq \binom{k}{2} - 2$. Also, $H'$ is connected because $H$ is $2$-connected. So by the induction hypothesis, we have $R(H',K_{n,n}) = O(n^{k-1})$. Now let $G$ be a graph on $N = Cn^k$ vertices with no $\overline{K_{n,n}}$. Then $G$ has no independent set of size $2n$. Hence, there exists $x \in V(G)$ with $d(x) \geq \frac{1}{2}Cn^{k-1} - 1$. By choosing $C$ large enough, we can make sure that $d(x) \geq R(H',K_{n,n})$. Then $G[N(x)]$ contains a copy of $H'$, which gives a copy of $H$ together with $x$. 
\end{proof}

\section{Concluding remarks and open problems}
\begin{itemize}
    \item It is worth mentioning an intriguing conjecture of Alon, Krivelevich and Sudakov \cite{Alon_Krivelevich_Sudakov}, that $R(H,K_n) \leq n^{O(r)}$ for every graph $H$ with $\Delta(H) \leq r$. Using the dependent random choice method, \cite{Alon_Krivelevich_Sudakov} showed that $R(H,K_n) = O(n^{(2r-k+2) \cdot (k-1)/2})$, where $k = \chi(H)$. So in the worst case $k = r$, the exponent is quadratic in $r$. The problem for $K_{n,n}$ (in place of $K_n$) turned out to be much easier and is resolved in Corollary \ref{cor:max degree}. 
    \item In Corollary \ref{cor:degenerate} we showed that $R(H,K_{n,n}) = O(n^{r+1})$ for an $r$-degenerate graph $H$. Can this be improved to $O(n^r)$? In particular, it would be very interesting to show that $R(H,K_{n,n}) = O(n^2)$ for every $2$-degenerate graph $H$.
    \item Balister et al. \cite{BSS} asked whether it is true that if $H$ is $2$-connected and has minimum degree $3$, then $H$ is {\bf not} Ramsey size-linear. A recent result of Janzer \cite{Janzer} gives a negative answer to this question. Indeed, \cite{Janzer} constructed $2$-connected $3$-regular bipartite graphs $H$ which have Tur\'an number at most $O(n^{3/2})$ (in fact, at most $O(n^{4/3+\varepsilon})$). Erd\H{o}s et al.~\cite{EFRS} observed that a bipartite graph $H$ with Tur\'an number at most $O(n^{3/2})$ is Ramsey size-linear. Hence, the graphs of \cite{Janzer} are Ramsey size-linear. 
    \item Theorem \ref{thm:vs bipartite} implies that $R(K_4^*,K_{n,n}) = O(n^2)$. For $K_n$, it is not difficult to prove that $R(K_4^*,K_n) = O(n^{5/2})$.
    Indeed, suppose that $G$ has $N = Cn^{5/2}$ vertices and no independent set of size $n$. 
    We may assume that $\delta(G) \geq \Omega(N/n)$, and then the average degree inside each neighbourhood is $\Omega(N/n^2)$. Also, each neighbourhood is $C_4$-free or else $G$ contains $K_4^*$ and we are done. It follows that there are at least 
    $N \cdot \Omega(N/n) \cdot \Omega(N/n^2)^2 \geq cN^4/n^5$ 4-tuples $x,y,z,w$ with $x \sim y,z,w$ and $y \sim z,w$. Here $c$ is some small absolute constant. On the other hand, the number of such 4-tuples $x,y,z,w$ with $d(z,w) \leq 2$ is at most $4\binom{N}{2} \leq 2N^2$. For $N = Cn^{5/2}$ with large enough $C$ (compared to $c$), we have $cN^4/n^5 > 2N^2$, so there is a 4-tuple $x,y,z,w$ with $d(z,w) \geq 3$. This gives a $K_4^*$. 
    It would be interesting to reduce the exponent $5/2$, hopefully all the way to $2$.
\end{itemize}

\paragraph{Acknowledgments:} The authors thank the anonymous referee for their careful reading of the paper and useful suggestions which improved the presentation.

\end{document}